\documentclass[12pt,DIV=13]{scrartcl}

\usepackage[utf8]{inputenc}
\usepackage[T1]{fontenc}
\usepackage{lmodern}

\usepackage[british]{babel}

\usepackage{amsmath,amsthm,amssymb,amsfonts,graphicx,mathtools,aliascnt}
\usepackage[shortlabels]{enumitem}
\usepackage[dvipsnames]{xcolor}
\usepackage{subcaption}
\usepackage[numbers]{natbib}
\usepackage{doi}

\usepackage{hyperref} % must be last
\hypersetup{
  colorlinks=true,
  linkcolor=[named]{RedViolet},
  citecolor=[named]{OliveGreen},
  filecolor=[named]{Periwinkle},
  urlcolor=[named]{magenta} %default
}

\newcommand\dd{\mathrm{d}}

\newcommand\EE{\mathbb{E}}
\newcommand\RR{\mathbb{R}}

\newcommand{\newsstheorem}[2]{
  \newaliascnt{#1}{dummy}
  \newtheorem{#1}[#1]{#2}
  \aliascntresetthe{#1}
  \expandafter\def\csname #1autorefname\endcsname{#2}
}
\theoremstyle{plain}
  \newsstheorem{thm}{Theorem}
  \newsstheorem{corollary}{Corollary}
  \newsstheorem{lemma}{Lemma}
  \newsstheorem{proposition}{Proposition}
  
\makeatletter
\newenvironment{proofof}[1]{\par
  \pushQED{\qed}%
  \normalfont \topsep6\p@\@plus6\p@\relax
  \trivlist
  \item[\hskip\labelsep
        \bfseries
    Proof of #1\@addpunct{.}]\ignorespaces
}{%
  \popQED\endtrivlist\@endpefalse
}
\makeatother
  
\theoremstyle{remark}
  \newsstheorem{remark}{Remark}
\theoremstyle{definition}
  \newsstheorem{definition}{Definition}
  \newsstheorem{exmp}{Example}

\title{The equivalence of two tax processes}
\author{Dalal Al Ghanim\footnote{University of Manchester, Department of Mathematics, UK. \texttt{dalal.alghanim@postgrad.manchester.ac.uk}.} \and  Ronnie Loeffen\footnote{University of Manchester, Department of Mathematics, UK. \texttt{ronnie.loeffen@manchester.ac.uk}.} \and Alexander R.\ Watson\footnote{University College London, Department of Statistical Science, UK. \texttt{alexander.watson.@ucl.ac.uk}.}}

\begin{document}

\maketitle

\begin{abstract}
  We introduce two models of taxation,
  the latent and natural tax processes, which have both been used to represent loss-carry-forward taxation on the capital of an insurance company.
  In the natural tax process, the tax rate is a function of the current level of capital, whereas in the latent tax process, the tax rate is a function of the capital that would have resulted if no tax had been paid.
  Whereas up to now these two types of tax processes have been treated
  separately, we show that, in fact, they are essentially equivalent.
  This allows a unified treatment, translating results from one
  model to the other.
  Significantly, we solve the question of existence and uniqueness
  for the natural tax process, which is defined via
  an integral equation.
  Our results clarify the existing literature on processes with tax.
\end{abstract}



\section{Introduction and main results}

Risk processes are a model for the evolution in time of the (economic) capital or surplus of an insurance company. 
Suppose that we have some model $X=(X_t)_{t\geq 0}$ for the risk process,
in which $X_t$ represents the capital of the company at time $t$;
for instance, a common choice is for $X$ to be a L\'evy process with negative
jumps.
Any such model can be modified in order to incorporate desirable features. 
For instance, reflecting the path at a given barrier models the situation where the insurance company pays out any capital in excess of the barrier as dividends to shareholders. Similarly, `refracting' the path at a given level and with a given angle corresponds to the case where dividends are paid out at a certain fixed rate whenever the capital is above the level or, equivalently, corresponds to a two-step premium rate. %, see e.g. Section VIII.1a in  \cite{albrecher_asmussen}. 
These modifications are described in more detail in Chapter 10 of \citet{FluctuationBook}, in the L\'evy process case.

Between the reflected and refracted  processes are a class of processes where partial reflection occurs whenever the process reaches a new maximum. The motivation in risk theory for these processes is that the times of partial reflection can be understood to correspond to tax payments associated with a so-called loss-carry-forward  regime in which taxes are paid only when the insurance company is in a profitable situation. In this paper we study \emph{tax processes} of this kind. 

%\medskip

Before we define rigorously the type of tax processes that we are interested in, we make some assumptions on $X$ that are in place throughout the paper.
We assume that $X$ is a stochastic process with c\`adl\`ag paths (i.e.,
right-continuous paths with left-limits) and without upward jumps
(that is, $X_t-\lim_{s\uparrow t}X_s\leq 0$ for all $t\geq 0$).
We also assume $X_0=x$ for some fixed $x\in\mathbb R$.

For example, these conditions are
satisfied if $X$ is a L\'evy process without upward jumps.
In fact, the main results presented in this work hold pathwise,
in the sense that they apply to each individual path of the stochastic process.
A random model is strictly only required for the study of specific examples;
however, given the applications we have in mind, it seems appropriate to
phrase everything in terms of stochastic processes.

\medskip
\noindent
One way to incorporate  a loss-carry-forward taxation regime for the risk process $X$ is to introduce
the tax process $U^\gamma \coloneqq (U_t^\gamma)_{t\geq 0}$ with
\begin{equation}\label{U}
  U_t^\gamma = X_t - \int_{0^+}^t \gamma(\overline{X}_s) \, \dd \overline{X}_s, \qquad t\geq 0,
\end{equation}
where $\gamma \colon [x,\infty) \to [0,1)$ is a measurable function and
$\overline{X}_t = \sup_{s\le t} X_s$ is the running maximum of $X$.
Note that, here and later,
$\int_{0^+}^t = \int_{(0,t]}$ denotes the integral over $(0,t]$. Since every path $t\mapsto\overline X_t$ is increasing (in the weak sense),  and is further continuous due to the assumptions on $X$, the integral in \eqref{U} is a well-defined Lebesgue-Stieltjes integral.
We call $U^\gamma$ a \emph{latent tax process} or the tax process with \emph{latent tax rate} $\gamma$. For this latent tax process we have that, roughly speaking, in the time interval $[t,t+h]$ with $h>0$ small, a fraction $\gamma(\overline X_t)$ of the increment $\overline X_{t+h}-\overline X_t$ is paid as tax. In particular, tax contributions are made whenever $X$ reaches a new maximum (which is whenever $U^\gamma$ reaches a new maximum; see Lemma \ref{equality in time for U and X} below), which is why the taxation structure in \eqref{U} can be seen to be of the loss-carry-forward type.
Since $\gamma < 1$, this can be seen as partial reflection;
setting $\gamma = \mathbf{1}_{[b,\infty)}$ would correspond to fully reflecting the path at the barrier $b$.

A great deal of literature has emerged in the study of this tax process.
It was introduced by \citet{LundbergRiskProcessWithTax}
in the case where $X$ is a Cram\'er--Lundberg process and $\gamma$ is a constant,
and in that work the authors studied the ruin probabilities, proving a strikingly simple
relation between ruin probabilities with and without tax, the so-called tax identity.
This work was extended by Albrecher, Renaud and Zhou \cite{LevyInsuranceRiskProcessWithTax}, using excursion theory, to the case
where $X$ is a general spectrally negative L\'evy process, with $\gamma$
still constant. In \cite{GeneralTaxStructure}, Kyprianou and Zhou
took $\gamma$ to be a function, and studied problems related to the two-sided
exit problem and the net present value of the taxes paid before ruin. In the same setting Renaud  \cite{TheDistributionOfTaxPayments} provided results on the distribution of the (present) value of the taxes paid before ruin. \citet{OptimalLossCarryForwardWang} studied a problem of optimal control of latent tax processes in which one seeks to maximise the net present value of the taxes paid before ruin.
%%%
A variation of the latent tax process in which the tax rate exceeds the value 1 can be found in \citet{KaprianouAndOtt}.

\medskip\noindent
An unusual property of the process $U^\gamma$ is that the taxation at time $t$ depends,
not on the running maximum $\overline{U}^\gamma_t = \sup_{s\le t} U^\gamma_s$
of the process $U^\gamma$ itself, but on the running maximum of $X$, i.e., $\overline{X}_t$. In other words, the amount of tax the company pays out at time $t$ is not determined by the amount of capital the company has at that time but it depends on a latent capital level, namely $\overline{X}_t$, which is the amount of capital that the company would have at time $t$ if no taxes were paid out at all. Besides being somewhat unnatural, this also means that in the common case where $X$ is modelled by a Markov process, the process $(U^\gamma,\overline{U}^\gamma)$ is not
Markov (see the first paragraph of Section \ref{s:remarks}).
In order to maintain the Markov property,
one would need to consider to the three-dimensional process $(U^\gamma,\overline{U}^\gamma,\overline{X})$.

For these reasons, it may be more
suitable to use another tax process $V^\delta=(V^\delta_t)_{t\ge 0}$, satisfying the equation
\begin{equation}\label{V}
  V^\delta_t = X_t - \int_{0^+}^t \delta(\overline{V}^\delta_s)\,\dd \overline{X}_s,
\end{equation}
where $\overline{V}^\delta_t = \sup_{s\le t}V^\delta_s$ and
$\delta\colon [x,\infty) \to [0,1)$ is a measurable function. % tax rate, which represents some taxed proportion of income depending on the capital. 
 We call $V^\delta$ a \emph{natural tax process} or a tax process with \emph{natural tax rate} $\delta$.
%We call $\delta$ a \emph{natural tax rate} and $V^\delta$ a \emph{natural tax process}.
Since \eqref{V} is an integral equation, it is not immediately clear whether
such a process $V^\delta$ exists and if so if it is uniquely defined. We will shortly give a simple condition for existence and uniqueness. Assuming that existence and uniqueness holds and that $X$ is a Markov process, the natural tax process $V^\delta$ has the advantage that the two-dimensional process $(V^\delta,\overline{V}^\delta)$ is Markov. For the reason that we retain $\overline{X}$ as the integrator in \eqref{V}, instead of using $\overline{V}^\delta$, see the second paragraph of Section \ref{s:remarks}.

Albrecher, Borst, Boxma and Resing \cite{TheTaxIdentity} looked at tax processes with a natural tax rate in the case where $X$ is a Cram\'er-Lundberg risk process and studied the ruin probability, though they do not provide a definition of the tax process in terms of an integral equation and in particular do not discuss existence and uniqueness. 
In the setting where $X$ is a Cram\'er-Lundberg risk process, \citet{wei} and \citet{CheungCramer} considered a more general class of natural tax processes than ours in which the associated premium rate is allowed to be surplus-dependent. % and they found analytic representations for several quantities of interest.
Although \cite{wei} and \cite{CheungCramer} do contain the definition \eqref{V} for the natural tax process in the case where $X$ is a Cram\'er-Lundberg risk process (see \cite[Section 1]{wei} with $\delta=0$ and \cite[Equation (1.2)]{CheungCramer} with the function $c(\cdot)$ being constant), neither paper addresses
the question of existence and uniqueness.

\medskip\noindent
The purpose of this work is to clarify the relationship between these two tax processes. Whereas latent and natural tax processes appear quite different when considering their definitions, it emerges that these two classes of tax processes are essentially equivalent, an observation which has seemingly gone unnoticed in the literature. This equivalence allows us to deal in a rather straightforward way with the existence and uniqueness of the natural tax process,  which is something that has not been dealt with before.

%and to derive identities for the two-sided exit problem for $V^\delta$. As a corollary, we will explain a curious result of \citet{OptimalLossCarryForwardWang} on optimal control of latent tax processes, and demonstrate that our approach is consistent with that of \citet{TheTaxIdentity}.

Before presenting our main theorem, we emphasise that our results hold true for a large class of stochastic processes for $X$ that includes, amongst others, spectrally negative L\'evy processes, spectrally negative Markov additive processes (see \cite{MAPtax2014}), diffusion processes (see \cite{litangzhou}) and fractional Brownian motion. However, practically, \eqref{U} and \eqref{V} may not in all cases be the right way to define a taxed process. For instance, when one considers a Cram\'er-Lundberg risk process where the company earns interest on its capital as well as pays tax according to a loss-carry-forward scheme, then one should not work with a process of the form \eqref{U} or \eqref{V}, but define the tax process differently, as in \cite{wei}. Our definitions \eqref{U} and \eqref{V} are practically suitable for modelling tax processes when the underlying risk process without tax  $X$ has a spatial homogeneity property, which is the case for, for instance, spectrally negative L\'evy processes, spectrally negative Markov additive processes or Sparre Andersen risk processes.

%We now summarise our results precisely. As stated, let $X$ be a spectrally negative L\'evy process, with associated probabilities $(\PP_x)_{x\in\RR}$ such that $X_0 = x$ $\PP_x$-almost surely.

In order to present our main result, we will need to consider
the following ordinary differential equation, for a given
measurable function $\delta \colon [x,\infty) \to [0,1)$:
\begin{equation}
  \label{number of the ODE}
  \begin{split}
    \dfrac{\dd y^\delta_{x}(t)}{\dd t} &= 1- \delta \left( y^\delta_{x}(t) \right), \qquad t\geq 0, \\
    y^\delta_{x}(0)&=x.
  \end{split}
\end{equation}
We say that $y^\delta_x \colon [0,\infty) \to\RR$
is a \emph{solution} of this ODE if it is an absolutely
continuous function and satisfies \eqref{number of the ODE} for
almost every $t$.

\begin{thm}\label{equivalence relation between U and V}
 Recall that $X_0 = x$. %; that is, we are working under measure $\PP_x$.
\begin{enumerate}[(i)]
		\item \label{i:intro:I}
		Let $U^\gamma$ be the tax process with latent tax rate $\gamma$, where $\gamma \colon [x,\infty)\rightarrow [0,1)$ is a measurable function.
		Define $\bar{\gamma}_{x} \colon [x,\infty)\to \RR$ by
		\begin{align}\label{GammaBar function}
		\bar{\gamma}_{x}(s) = x + \int_{x}^{s}(1-\gamma(y)) \dd y, \qquad s \geq x,
		\end{align}
		and consider its inverse
		$\bar{\gamma}_x^{-1} \colon [x,\infty] \to [x,\infty]$,
		with the convention that $\bar{\gamma}_x^{-1}(s) = \infty$
		when $s\ge \bar{\gamma}_x(\infty)$.
		Define $\delta_{x}^{\gamma}\colon [x,\bar{\gamma}_{x}(\infty)) \rightarrow[0,1)$ by
		$\delta_{x}^{\gamma}(s) = \gamma (\bar{\gamma}_{x}^{-1}(s))$.
		Then, 
		\begin{equation}\label{Ubar as a function of Xbar}
		\overline{U}^\gamma_t = \bar\gamma_x(\overline X_t),
		\qquad t \geq 0,
		\end{equation}
		 and $U^\gamma$ is a natural tax process with natural tax rate $\delta^\gamma_x$.
	\item\label{i:intro:II}
	Let $\delta \colon [x,\infty)\rightarrow [0,1)$ be a measurable function and assume that there exists a unique solution $y^\delta_{x}(t)$	of \eqref{number of the ODE}. Define $\gamma_{x}^{\delta}\colon [x, \infty)\to [0,1)$   by $\gamma_{x}^{\delta}(s) =  \delta \left( y^\delta_{x}(s-x) \right)$.  
	
	Then, the integral equation \eqref{V} defining the natural tax process
	has a unique solution $V^\delta=(V^\delta_t)_{t\geq 0}$. Moreover, 
	\begin{equation}\label{Vbar=Y(Xbar-x)}
	\overline{V}^{\delta}_{t}= y^\delta_{x}(\overline{X}_{t}-x), \qquad t \geq 0,
	\end{equation}
	and so the solution $V^\delta$ to \eqref{V} is a latent tax process with latent tax rate given by $\gamma_{x}^{\delta}$.
\end{enumerate}
\end{thm}

This theorem is the main contribution of the article. It states that a sufficient condition for existence and uniqueness of solutions to \eqref{V}
can be given in terms of a simple ODE. From the proofs given in Section \ref{s:prelim and proofs} below, it is not difficult to see that the existence and uniqueness of the ODE \eqref{number of the ODE} is also a necessary condition for existence and uniqueness of a solution to \eqref{V}. Theorem \ref{equivalence relation between U and V} also gives a precise relationship between the two types of tax processes.
In particular, every latent tax process is a natural tax process, though the corresponding latent and natural tax rates may differ. Conversely, every
well-defined natural tax process is also a latent tax process.
The next example illustrates this equivalence
for piecewise constant tax rates.

\begin{exmp}\label{example}
%Let $X$ be a L\'evy process started from $X_0 = x$,
Define the piecewise constant function $f^b$ by
\begin{equation} \label{e:fb}
f^{b}(z)= \begin{cases} \alpha, &   z \leq b,\\
\beta, &z > b,
\end{cases}
\end{equation}
where $b > x=X_0$ and $0 \leq \alpha \leq \beta < 1$. Note that the ODE \eqref{number of the ODE} with $\delta=f^b$ has a unique solution, see e.g. Section \ref{s:remarks}.
It is clear that the tax process  with latent tax rate $f^b$ differs from the tax process with natural tax rate  $f^b$,
unless $\alpha = \beta$ or $\alpha=0$. However, from \autoref{equivalence relation between U and V} we deduce that the tax process with latent tax rate $f^b$ is equal to the tax process with natural tax rate $f^{b'}$  for
\begin{equation*}
 b' =  (1-\alpha)b + \alpha x.
\end{equation*}  
Note that $b'$ depends on the starting point $x$, unless $\alpha=0$.  \autoref{figure A} contains two plots in which an example of $X$ and the corresponding tax process $U^{f^b}$, or equivalently $V^{f^{b'}}$, are drawn. From this figure we see that indeed the first time $X$ reaches the level $b$ is equal to the first time the tax process reaches the level $b'$. %The equality of these two tax processes is further illustrated in \autoref{figure A}.
\end{exmp}

\begin{figure}
	% PDF files generated by exporting the Matlab figure as EPS and using
	% epstopdf to convert. Matlab is still, in the year 2018, incapable of
	% exporting a PDF with an appropriate bounding box.
	\centering
	\begin{subfigure}[a]{\textwidth}
		\centerline{\includegraphics[width=1.2\textwidth]{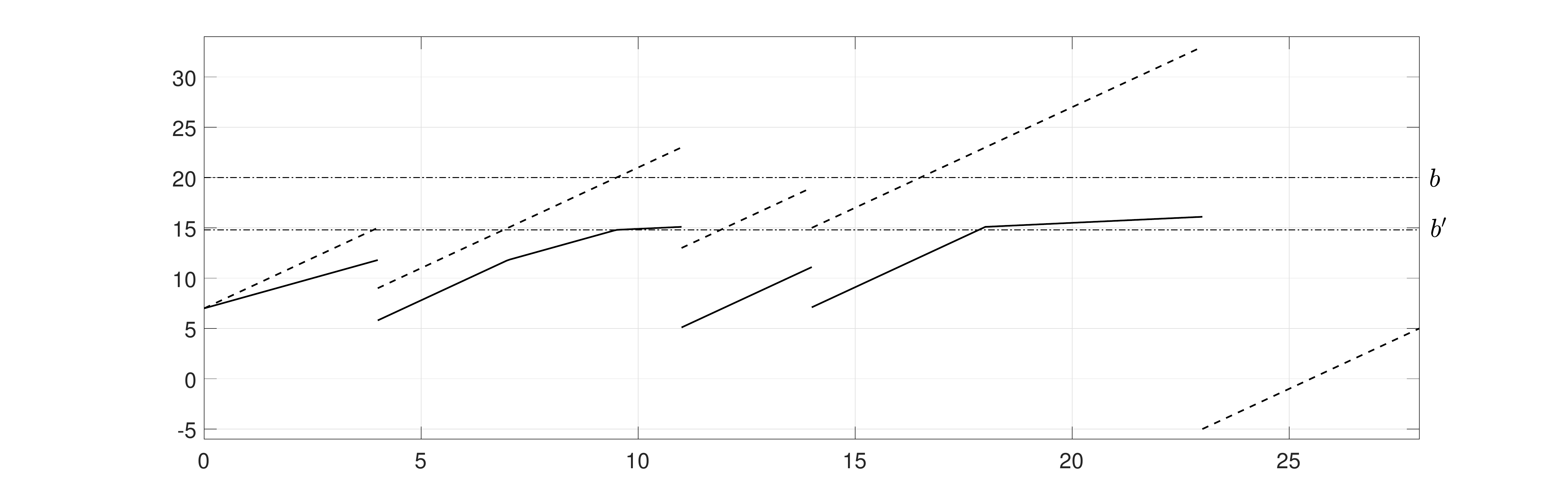}}
		\caption{$x=7$, $b = 20$ and $b' = 14.8$.}
		\label{figure 1} 
	\end{subfigure}
	
	\begin{subfigure}[b]{\textwidth}
		\centerline{\includegraphics[width=1.2\textwidth]{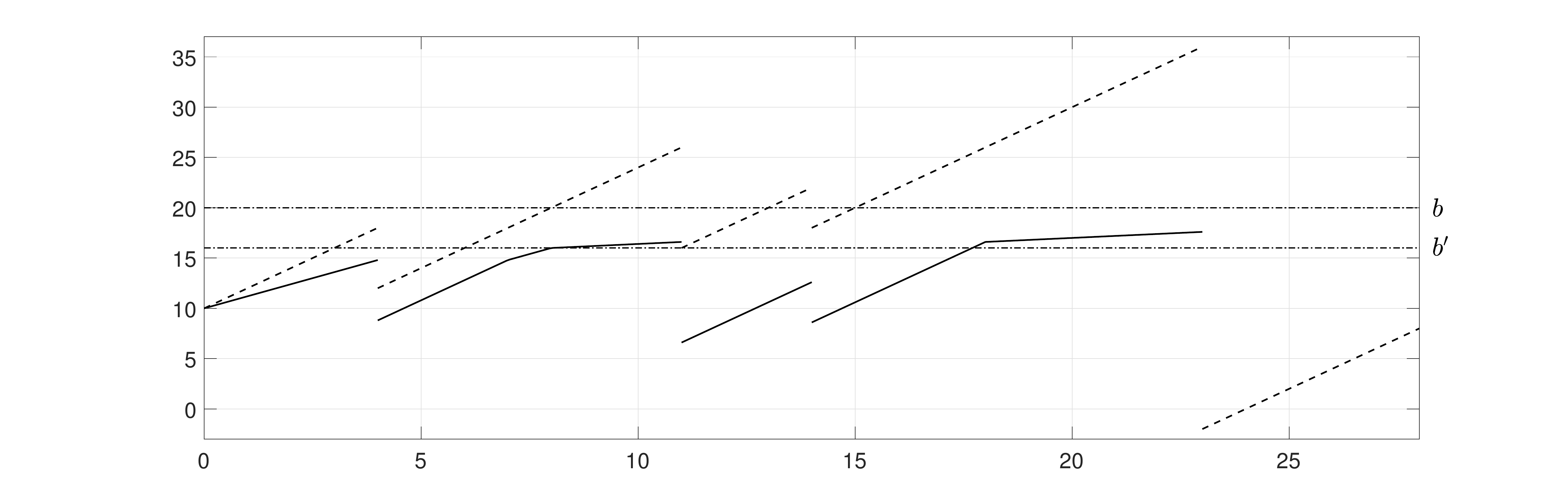}}
		\caption{$x=10$, $b = 20$ and $b' = 16$.}
		\label{figure 2}
	\end{subfigure}
	
	\caption[No table of figures.]{%
		% Paragraphs are too bunched up by default.
		%\setlength{\parskip}{0.3\baselineskip plus2pt minus2pt}%
		Plots of the risk process $X$ (dashed line) and the associated latent tax process $U^{f^b}$ or equivalently natural tax process $V^{f^{b'}}$ (solid line), where $f^b$ is the piecewise constant function defined by \eqref{e:fb} with $\alpha = 0.4$ and $\beta = 0.9$.
		The dashed-dot lines mark  the values of $b$ and $b^{\prime}$.
		\label{figure A}
	}
\end{figure}

The theorem allows us to very easily translate results derived for the latent tax process to results on the natural tax process, or vice versa. As an example, by using the corresponding result derived in \cite{GeneralTaxStructure} for the latent tax processes, we provide below an analytical expression of the so-called two-sided exit problem of the natural tax process in the case where $X$ is a spectrally negative L\'evy process. %, i.e. $X$ is a stochastic process with stationary and independent increments, having c\`adl\`ag sample paths and exhibiting no upward jumps. 
%The expression we provide is consistent with Equation (2.24) in \cite{CheungCramer} for the special case where $X$ is a Cram\'er-Lundberg risk process. 
For an introduction to spectrally negative L\'evy processes and their scale functions we refer to Chapter 8 in \cite{FluctuationBook}.

%For instance, we obtain the following identity by applying results of \cite{GeneralTaxStructure}. \arw{Scrap this? -- Note that, in \cite{GeneralTaxStructure}, the results were given	in the case $\bar{\gamma}_x(\infty) = \infty$, but this assumption	proves inessential.} Here, $W^{(q)}$ denotes the $q$-scale function of the Lévy process $X$, which we define precisely in section \ref{s:prelim and proofs}.

\begin{corollary}\label{t:identities}
	Let $X$ be a spectrally negative L\'evy process on the probability space $(\Omega,\mathcal F,\mathbb P_x)$ such that $\mathbb P_x(X_0=x)=1$. 	Let $\delta \colon [x,\infty)\rightarrow [0,1)$ be a measurable function such that there exists a unique solution $y^\delta_{x}$ to \eqref{number of the ODE}. Let $V^\delta$ be the tax process with natural rate $\delta$ associated with the spectrally negative L\'evy process $X$.
	%Suppose we have the taxed process
	%	\begin{align*}
	%	V^{\delta}_{t} = X_{t} - \int^{t}_{0^+} \delta(\overline{V}^{\delta}_{r}) \dd \overline{X}_{r},\ \ X_{0}=x.
	%	\end{align*}
	Define the first passage times
	\begin{equation*}
	%\label{two sided exit problem for delta V}
	\tau_a^- = \inf\{t\ge 0 : V^{\delta}_t < a\}
	\quad \text{ and } \quad 
	\tau_a^+ = \inf\{t\ge 0 : V^{\delta}_t > a\},
	\end{equation*} 
	where $a \in \RR $. Let $q\geq 0$ and let $W^{(q)}\colon \mathbb R\to[0,\infty)$ be the \emph{$q$-scale function} of $X$, defined by $W^{(q)}(z)=0$ for $z<0$ and characterised on $[0,\infty)$ as the continuous function whose Laplace transform  is given by
	\begin{equation*}
	\int_0^\infty e^{-\lambda y} W^{(q)}(z)\, \mathrm d z = \left(  \log \left( \mathbb E \left[ e^{\lambda X_1} \right] \right) -q  \right)^{-1}, \qquad \text{for $\lambda>0$ sufficiently large}.
	\end{equation*}
	%%%%%%%%%%%%5
	Then, for $0\leq x<a< y^\delta_{x}(\infty)$, we have
	\begin{align}
	\label{formula of two sided exit problem for delta V}
	\mathbb{E}_{x} \left[ e^{-q \tau^{+}_{a} } \mathbf{1}_{\left\lbrace \tau^{+}_{a} < \tau^{-}_{0} \right\rbrace }\right]  =
	\exp \left\lbrace  - \int^{a}_{x}\dfrac{W^{(q)\prime}(y)}{W^{(q)}(y)(1- \delta(y))} \, \dd y \right\rbrace,
	\end{align}
	where $W^{(q)\prime}$ denotes a density of $W^{(q)}$ on $(0,\infty)$.
	On the other hand, if $a\geq y^\delta_{x}(\infty)$, then
	$\mathbb{E}_{x} \left[ e^{-q \tau^{+}_{a} } \mathbf{1}_{\left\lbrace \tau^{+}_{a} < \tau^{-}_{0} \right\rbrace }\right]=0.$
	
\end{corollary}

The rest of this article is organised as follows. In Section \ref{s:remarks}, we explain the consequences of our results and make further connections with the literature. %the comparison with \cite{OptimalLossCarryForwardWang} and \cite{TheTaxIdentity}.
Section \ref{s:prelim and proofs} is devoted to the proofs of Theorem \ref{equivalence relation between U and V} and Corollary \ref{t:identities}. 

\section{Related work and further applications}
\label{s:remarks}

\paragraph{Markov property}
Assume $X$ is a Markov process.  As we have already commented in the previous section, it follows from the  integral
  equation \eqref{V} for the natural tax process $V^{\delta}$ that the process $(V^\delta,\overline{V}^\delta)$ is  Markov.  One might expect that the equivalence between the two types of tax processes should imply the same for $(U^\gamma,\overline{U}^\gamma)$ where $U^\gamma$ is an arbitrary latent tax process, since we know by Theorem \ref{equivalence relation between U and V}\ref{part (I)}  that $U^\gamma$ is also a natural tax process.    However, the corresponding natural tax rate is $\delta^\gamma_x = \delta^\gamma_{X_0}$, which 
  depends  upon the initial value of $X$.
  Looked at another way, although one can recover
  $\overline{X}$ from the formula
  $\overline{X}_{t}= \bar{\gamma}_{x}^{-1}(\overline{U}_{t}^\gamma)$,
  this too depends on knowledge of the initial value $X_0$.
  For this reason, we do not obtain the  Markov property for $(U^\gamma,\overline{U}^\gamma)$ in general.

\paragraph{An alternative definition of the natural tax process}
It would also appear to be reasonable to define a natural type of tax process
as a solution to the SDE % $W^\kappa=(W^\kappa_t)_{t\ge 0}$ with
\begin{equation}\label{e:W}
  W_t = X_t - \int_{0^+}^t \kappa(\overline{W}_s) \,\dd \overline{W}_s,
\end{equation}
where $\kappa \colon [x,\infty) \to [0,\infty)$.
Define 
$\delta = \frac{\kappa}{1+\kappa}$.
The process $V^\delta$, when it exists, is a solution to \eqref{e:W},
as can be shown using
Lemma \ref{equality in time for U and X}.
The
natural tax rate $\delta$ describes the tax rate as a proportion of the increments
of capital prior to taxation rather than after taxation, and therefore appears
to us to be preferable to $\kappa$ as a parameter.

\paragraph{Existence of the tax process with progressive natural tax rates}
When the tax rate increases with the amount of capital one has, the taxation regime is typically called progressive. We will show that, when $\delta$ is an increasing (in the weak sense) measurable function $\delta \colon [x,\infty) \rightarrow [0,1)$, then the ODE \eqref{number of the ODE} has a unique solution, which implies the existence and uniqueness of the natural tax process with tax rate $\delta$.

For existence, since $\delta$ is an increasing function, we have that
\[g(z) \coloneqq \dfrac{1}{1- \delta(z)}, \qquad z\ge x, \]
is a strictly positive, increasing measurable function, and hence integrable, so
\[G(y) \coloneqq \int^{y}_{x} g(z)\, \dd z, \qquad y \ge x, \]
is absolutely continuous.  Moreover, since $G$ is continuous and strictly increasing, $G^{-1}$ exists and, as $G^{\prime}>0$ a.e., $G^{-1}$ is absolutely continuous \cite[Vol.~I, p.~389]{MeasureTheory}. Thus, $(G^{-1})^{\prime}(t)$ exists for almost every $t$, and it follows that a solution to \eqref{number of the ODE} is given by $y_{x}(t) = G^{-1}(t)$. This is because, by the inverse function theorem \cite[Theorem 31.1]{MultivariableAnalysis}, it holds that
\[\dfrac{\dd G^{-1}(t)}{\dd t} = \frac{1}{g(G^{-1}(t))} = 1- \delta(G^{-1}(t)), \qquad \text{for a.e. } t > 0,\]
and since $G(x) = 0$, we have $G^{-1}(0) = x$.
 
For uniqueness, since $\delta$ is increasing, the right hand side of \eqref{number of the ODE} is decreasing. % and hence satisfies the one sided Lipschitz condition. 
 This guarantees uniqueness, as can be proved using, for instance, \cite[Theorem 1.3.8]{GeneralLinearMethodsForODEs}.  
% \end{remark}

\paragraph{Optimal control}
In the case where $X$ is a spectrally negative L\'evy process, \citet{OptimalLossCarryForwardWang} studied a very interesting optimal control problem for the latent tax process $U^\gamma$, given by
\begin{equation}\label{e:WHoc}
  \sup_{\gamma \in \Pi}
  \EE_x\biggl[ \int_0^{\sigma_0^-} e^{-qt} \gamma(\overline{X}_t) \, \dd\overline{X}_t
  \biggr],
\end{equation}
where $\sigma_0^- = \inf\{t\ge 0: U^{\gamma}_t < 0 \}$, and $\Pi$ is
the set of measurable functions $\gamma \colon [0,\infty) \to [\alpha,\beta]$,
where $0\le \alpha \leq \beta < 1$ are fixed.
Denote by $\gamma^*$ the function $\gamma\in\Pi$ which maximises
\eqref{e:WHoc}, if it exists.
A remarkable feature of Wang and Hu's work is that they obtain a
\emph{natural} tax process as the optimal solution to the problem
of controlling a latent tax process, as we will now explain.

In their Theorem 3.1, Wang and Hu state that  $\gamma^*$ should satisfy
the equation
\begin{equation*}
  \gamma^{*}(\overline{X}_{t})
  = \eta \left( x + \int_{x}^{\overline{X}_{t}} (1 - \gamma^{*}(y) ) \, \dd y \right)
  = \eta (\overline{U}_{t}^{\gamma^*}),
\end{equation*}
for some function $\eta$ which they call the \emph{optimal decision rule}.
On the other hand, let $\delta$ be a function satisfying the assumptions of \autoref{equivalence relation between U and V}\ref{i:intro:II}, and define $\gamma^\delta_x$ as in that result. If we write $\xi = \gamma^{\delta}_x$,  then, by the definition of $\gamma^{\delta}_x$ together with \eqref{Ubar as a function of Xbar} and \eqref{Vbar=Y(Xbar-x)}, we have the relation
\begin{equation*}
  \xi(\overline{X}_t) = \delta \left( x + \int_{x}^{\overline{X}_{t}} (1 - \xi(y) ) \, \dd y \right)  =  \delta(\overline{U}_{t}^\xi).
\end{equation*}

It follows from \autoref{equivalence relation between U and V} that the relationship between Wang and Hu's optimal decision rule $\eta$ and optimal tax rate $\gamma^*$  is nothing other than the relationship between a particular natural tax rate $\delta$ and the equivalent latent tax rate $\gamma_x^\delta$. Our results clarify that this connection is a sensible one
even outside of the optimal control context, and make clear under exactly which conditions this connection is valid.

Wang and Hu go on to show that $\eta$ must be piecewise constant,
and in particular
$\eta = f^b$, as defined in \eqref{e:fb}, where $b$ is specified in terms of
scale functions of the L\'evy process but is independent of $x$;
see section 4 and equation (5.15) in their work (in which $b$ is denoted $u_0$).
Combining this with our result, 
we see that Wang and Hu's solution of the optimal control
problem \eqref{e:WHoc} is actually a tax process with the piecewise constant natural tax rate
$f^b$, or equivalently the piecewise constant latent tax rate $f^{\tilde b(x)}$, where $\tilde b(x)$ depends on $x$ as in \autoref{example}.

\paragraph{Tax identity}

Assume we are in the setting of Corollary \ref{t:identities} where in particular $X$ is a spectrally negative L\'evy process. We are interested here in the tax identity: a relationship between the survival probability of the natural tax process $V^\delta$ and the one of the risk process with out tax $X$. To this end, let
\begin{equation*}
\phi_{\delta}(x) = \mathbb{P}_{x} \left(  \inf_{t \geq 0} V^{\delta}_{t} \geq 0 \right) \quad
\text{and} \quad 
\phi_{0}(x) = \mathbb{P}_{x} \left(  \inf_{t \geq 0} X_{t} \geq 0 \right)
\end{equation*}
be the survival probability in the risk model with and without taxation, respectively.

If $y^\delta_{x}(\infty) < \infty$,
  the process $V^\delta$ cannot exceed the level $y^\delta_{x}(\infty)$.
  Since from every starting level (and thus in particular from $y^\delta_{x}(\infty)$), there is a strictly positive probability  of $V$ going below zero, a standard renewal argument shows that  the survival probability $\phi_\delta(x)$ is zero in this case.
  
  On the other hand, if $y_{x}(\infty) = \infty$,  then we can apply Corollary \ref{t:identities} to get a relation between the two survival probabilities. Namely,
  by letting $q\rightarrow 0$ and $a\rightarrow\infty$ in \eqref{formula of two sided exit problem for delta V} and using the well-known expression for $\phi_{0}(x)$ (see, e.g., \cite[equation (8.10)]{FluctuationBook}), we have that
  \begin{align*}
  %\label{survival probability of V}
  \phi_{\delta}(x)
  =  \exp \left\lbrace  - \int^{\infty}_{x}\dfrac{W^{\prime}(y)}{W(y)(1- \delta(y))} \dd y \right\rbrace 
  = \exp \left\lbrace  - \int^{\infty}_{x} \dfrac{\dd \ln (\phi_{0}(y))}{\dd y} \cdot \dfrac{1}{(1- \delta(y))} \dd y \right\rbrace.
  \end{align*}
  This agrees with \cite[Proposition 3.1]{TheTaxIdentity} for the special case where $X$ is a Cram\'er-Lundberg risk process, which confirms that in \cite{TheTaxIdentity} natural tax processes are considered.

\section{Proofs}
\label{s:prelim and proofs} 

We start with a lemma generalising a result from \cite{GeneralTaxStructure}.
\begin{lemma}\label{equality in time for U and X}
Let $K = (K_t)_{t\ge 0}$ be a stochastic process for which every path is measurable
as a function of time and such that $K_t<1$ for every $t\ge 0$. Define
 \begin{equation*}%\label{process U}
 H_{t} = X_{t} - \int^{t}_{0^+} K_s  \dd \overline{X}_{s}, \qquad t\geq 0.
 \end{equation*}
  Then,
\begin{equation*}
\overline{H}_t = \overline{X}_{t} - \int^{t}_{0^+} K_s  \dd \overline{X}_{s},
\end{equation*} 
where  $\overline{H}_{t} = \sup_{s \leq t} H_{s}$. Moreover,
%\begin{equation*}
  $\lbrace t\geq 0:  {H}_{t} = \overline{H}_{t} \rbrace = \lbrace t\geq 0:  X_{t} = \overline{X}_{t} \rbrace.$
% \end{equation*}
\end{lemma}
\begin{proof}
  Since $K_t< 1$ for all $t\geq 0$, the proof in \cite[Lemma 2.1]{GeneralTaxStructure}
  works without alteration.
\end{proof}
 
Next we prove part \ref{part II}  of Theorem \ref{equivalence relation between U and V} except for the existence of the integral equation.
\begin{lemma}\label{l:char SDE}
	Let $\delta \colon [x,\infty)\rightarrow [0,1)$ be a measurable function and assume that there exists a unique solution $y^\delta_{x}$ of \eqref{number of the ODE}. Define $\gamma_{x}^{\delta}\colon [x, \infty)\to [0,1)$
  by $\gamma_{x}^{\delta}(s) =  \delta \left( y^\delta_{x}(s-x) \right)$.  
	If there exists a solution $V^\delta=(V_t^\delta)_{t\geq 0}$ to the integral equation
	\begin{equation}
	  \label{Vt SDE 2}
	  V^{\delta}_{t} 
	  = X_{t} - \int^{t}_{0^+} \delta(\overline{V}^{\delta}_{r}) \dd \overline{X}_{r}, \qquad t\geq 0,
	 % \qquad V^{\delta}_0 = x,
	\end{equation}
	then
	  $\overline{V}^{\delta}_{t}= y^\delta_{x}(\overline{X}_{t}-x)$ 
	and hence $V^\delta$  is a latent tax process with latent tax rate given by $\gamma_{x}^{\delta}$.
\end{lemma}
\begin{proof}
  Suppose that $V^{\delta}$ solves \eqref{Vt SDE 2}.
  By Lemma \ref{equality in time for U and X},
\begin{equation}\label{runningmax_V}
    \overline{V}^{\delta}_{t} 
    = \overline{X}_{t} - \int^{t}_{0^+} \delta(\overline{V}^{\delta}_{r})
    \, \dd \overline{X}_{r}, \qquad t\geq 0.
\end{equation}
 We define $L_t = \overline{X}_t - x$ and we let $L_a^{-1}$ be its right-inverse, i.e. 
 \begin{align*}
 L_{a}^{-1} \coloneqq
 \begin{cases} 
 \inf \lbrace t>0: L_{t}>a \rbrace = \inf\{t>0:\overline X_t > a+x\}, & \text{if } 0\leq a<L_{\infty}, \\
 \infty, & \text{if } a\geq L_{\infty}.
 \end{cases}
 \end{align*}
 As $X$ does not have upward jumps, $t\mapsto \overline X_t$ is continuous, which implies 
 \begin{equation}\label{maxX_at_inv}
 \overline{X}_{L^{-1}_{a}} = x+ (a \wedge L_\infty).
 \end{equation}
 Using respectively \eqref{runningmax_V} for $t=L^{-1}_a=L^{-1}_{a\wedge L_\infty}$, \eqref{maxX_at_inv} and  the change of variables formula with $r=L^{-1}_b$ (see, for instance, \cite[footer of p.~8]{revuzyor}), we have for $a\geq 0$,
 \begin{equation*}
 \begin{split}
\overline{V}^{\delta}_{L^{-1}_{a\wedge L_\infty}} & = \overline{X}_{L^{-1}_{a\wedge L_\infty}} - \int^{L^{-1}_{a\wedge L_\infty}}_{0^+} \delta(\overline{V}^{\delta}_{r}) \, \dd \overline{X}_{r} \\
& =  x + (a \wedge L_\infty)  - \int^{\infty}_{0^+} \mathbf 1_{ \left\{ r\leq L^{-1}_{a\wedge L_\infty} \right\} } \delta(\overline{V}^{\delta}_{r}) \, \dd \overline{X}_{r} \\
& =  x + (a \wedge L_\infty) - \int^{\infty}_{0} %\mathbf{1}_{\lbrace s < t \rbrace} 
\mathbf 1_{ \left\{ 0< L^{-1}_b \leq L^{-1}_{a\wedge L_\infty} \right\} } \delta(\overline{V}^{\delta}_{L^{-1}_{b}})\, \dd b \\
& =  x + \int^{a \wedge L_\infty}_{0} \left( 1 - \delta \left( \overline{V}^{\delta}_{L^{-1}_{b}} \right) \right) \, \dd b,
 \end{split}
 \end{equation*}
 where for the last equality we used that $L^{-1}_b$ is strictly increasing on $[0,L_\infty]$, which follows because $t\mapsto \overline X_t$ is continuous.
 By the hypothesis that \eqref{number of the ODE} has a unique solution $y_x^\delta$, we deduce,
 \begin{equation}\label{lemma_laststep}
  \overline{V}^{\delta}_{L^{-1}_{a\wedge L_\infty}} 
  = y^\delta_{x} \left( a \wedge L_\infty  \right) = y^\delta_{x} \left( \overline{X}_{L^{-1}_{a\wedge L_\infty}}-x \right), \qquad a\geq 0,
 \end{equation}
 where the last equality follows by  \eqref{maxX_at_inv}.  As $t\mapsto\overline X_t$ does not jump upwards, $\overline{X}_{L^{-1}_{L_t}} = \overline X_t$ for all $t\geq 0$, which implies via \eqref{runningmax_V} that $\overline V^\delta_{L^{-1}_{L_t}}=\overline V^\delta_t$ for all $t\geq 0$. So by invoking \eqref{lemma_laststep} for $a=L_t$, we conclude that  $\overline{V}^{\delta}_{t}= y^\delta_{x}(\overline{X}_{t}-x)$  for all $t\geq 0$.
\end{proof}
 
 We are now ready to prove Theorem \ref{equivalence relation between U and V} and Corollary \ref{t:identities}.
\begin{proofof}{Theorem \ref{equivalence relation between U and V}}
	\begin{enumerate}[(i)]
		\item \label{part (I)} Fix $t\ge 0$.
		By \autoref{equality in time for U and X}, we have
	\begin{equation*}
		  \overline{U}_{t}^{\gamma} = \overline{X}_{t} - \int^{t}_{0^+} \gamma(\overline{X}_{r}) \dd \overline{X}_{r}.
	\end{equation*}
	By applying the change of variable $y=\overline{X}_{r}$, we obtain
    \begin{equation*}
      \overline{U}_{t}^{\gamma} 
      = \overline{X}_{t} -\int^{\overline{X}_{t}}_{x} \gamma(y) \dd y 
      = \bar{\gamma}_{x}(\overline{X}_{t}),
    \end{equation*}
    where we recall that $\bar{\gamma}_x(s) = x + \int_x^s(1-\gamma(y))\,\dd y$.
    Hence, $\bar{\gamma}_{x}^{-1}(\overline{U}^{\gamma}_{t}) = \overline{X}_{t}$,
    and so
    $\gamma(\overline{X}_{t}) = \gamma(\bar{\gamma}_{x}^{-1}(\overline{U}^{\gamma}_{t})) 
    =  \delta_{x}^{\gamma}(\overline{U}^{\gamma}_{t})$.
    It follows that $U^{\gamma}$ is a natural tax process with natural tax rate $\delta^\gamma_x$.

  \item \label{part II}
  The uniqueness of a solution to \eqref{V}, and the equality \eqref{Vbar=Y(Xbar-x)}, follow directly from Lemma \ref{l:char SDE}. So it remains to prove the existence of a solution to \eqref{V}. 
  By the hypothesis there exists a unique solution $y^\delta_x$ to \eqref{number of the ODE}.
  With $\gamma_{x}^{\delta}(z)= \delta \left( y^\delta_{x}(z-x) \right)$, we  define $\bar{\delta}\colon [x,\infty)\rightarrow[0,1)$ by
  \begin{equation*}  
  \bar{\delta}(z) = \gamma_{x}^{\delta} \bigl( (\bar{\gamma}^{\delta}_{x})^{-1} (z)\bigr) = \delta \left( y^\delta_{x} \left( \left( \bar{\gamma}^\delta_{x} \right)^{-1}(z) - x \right) \right),
  \end{equation*}
  where $(\bar{\gamma}^{\delta}_{x})^{-1}$ is the inverse function of
  \begin{equation}\label{def_bargamdel}
  \bar{\gamma}^\delta_{x}(z) = x + \int_{x}^{z}(1-\gamma_{x}^{\delta}(y)) \dd y.
    %=: \bar{\gamma}^{\delta}_{x}(s)
  \end{equation}
  By part \ref{part (I)} of Theorem \ref{equivalence relation between U and V}, the tax process with latent tax rate $\gamma_{x}^{\delta}$ is a natural tax process with natural tax rate $\bar\delta$.  
  Thus, it remains to show that $\bar{\delta}(z) = \delta(z)$ for $z \geq x$.
  
  Note that $\bar{\gamma}^\delta_{x}$ is an absolutely continuous function and hence $(\bar{\gamma}_{x}^\delta)^{\prime}$ exists almost everywhere. By \eqref{number of the ODE} we have that, for $z$ such that $(\bar{\gamma}_{x}^\delta)^{\prime}(z)$ exists,
\begin{equation*}
\begin{split}
\frac{\dd}{\dd z} \left( y_{x}^\delta( (\bar{\gamma}^\delta_{x})^{-1}(z) - x) \right) 
&=  \left[ 1 - \delta \left( y^\delta_{x} \left( (\bar{\gamma}_{x}^\delta)^{-1}(z) - x \right) \right) \right] \dfrac{\dd}{\dd z} \left( (\bar{\gamma}_{x}^\delta)^{-1}(z) \right) \\
&=  \left[ 1-\gamma_{x}^{\delta} \left( (\bar{\gamma}^\delta_{x})^{-1}(z) \right) \right] \dfrac{\dd}{\dd z} \left( (\bar{\gamma}_{x}^\delta)^{-1}(z) \right). 
\end{split}
\end{equation*}
  Since by the inverse function theorem \cite[Theorem 31.1]{MultivariableAnalysis},
  \begin{align*}
    \dfrac{\dd }{\dd z} \left( (\bar{\gamma}^\delta_{x})^{-1}(z) \right)
    &= \frac{1}{(\bar{\gamma}^\delta_{x})^{\prime} \left( (\bar{\gamma}^\delta_{x})^{-1}(z) \right)} 
    = \dfrac{1}{ 1-\gamma_{x}^{\delta} \left( (\bar{\gamma}^\delta_{x})^{-1}(z) \right)  },
  \end{align*}
   we see that
\begin{equation*}
\frac{\dd}{\dd z} \left( y_{x}^\delta( (\bar{\gamma}^\delta_{x})^{-1}(z) - x) \right)  = 1 \qquad \text{a.e.},
\end{equation*}
  and therefore, by the absolute continuity, for some constant $c$, we have that
\begin{equation*}
y_{x}^\delta( (\bar{\gamma}^\delta_{x})^{-1}(z) - x)= z + c, \qquad z\geq x.
\end{equation*}
  Since $(\bar{\gamma}_{x}^\delta(x))^{-1}= x =y^\delta_{x}(0)$, we get that $c=0$. We conclude that
  $\bar\delta(z)=\delta(z)$ for $z\geq x$, and this completes the proof.
  \end{enumerate}
\end{proofof}

\begin{proofof}{\autoref{t:identities}}
 
    From \eqref{Vbar=Y(Xbar-x)} we see that $\tau_a^+=\infty$ when $a\geq y_{x}(\infty)$. Hence we can assume without loss of generality that  $a< y^\delta_{x}(\infty)$. By part \ref{part II} of Theorem \ref{equivalence relation between U and V} we know that $V^\delta$ is a latent tax process with latent tax rate $\gamma_{x}^{\delta}$. Hence we can use Theorem 1.1 in \cite{GeneralTaxStructure} to conclude that,
    \begin{equation*}
    \mathbb{E}_{x} \left[ e^{-q \tau^{+}_{a} } \mathbf{1}_{\left\lbrace \tau^{+}_{a} < \tau^{-}_{0} \right\rbrace }\right]  =
    \exp \left\lbrace  - \int^{a}_{x}\dfrac{W^{(q)\prime}(y)}{W^{(q)}(y) \left( 1-\gamma_{x}^{\delta} \left( (\bar{\gamma}_{x}^\delta)^{-1}(y) \right) \right) } \dd y \right\rbrace,
    \end{equation*}
    where $(\bar{\gamma}_{x}^\delta)^{-1}$ is the inverse of the function $\bar{\gamma}_{x}^\delta$ given by \eqref{def_bargamdel}. Note that in \cite{GeneralTaxStructure} the additional assumption $\int_0^\infty (1-\gamma_{x}^{\delta}(z))\mathrm d z=\infty$ is made on the latent tax rate, but from the proof of Theorem 1.1 in \cite{GeneralTaxStructure} it is clear that this assumption is unnecessary when $a< y^\delta_{x}(\infty)$.
    In the proof of Theorem \ref{equivalence relation between U and V}\ref{part II} we showed that $\gamma_{x}^{\delta} \left( (\bar{\gamma}_{x}^\delta)^{-1}(y) \right) =\delta(y)$ for all $y\geq x$, which finishes the proof. 
\end{proofof}

\paragraph{Acknowledgements} 
We thank two anonymous referees for their helpful comments. Also, this work has been funded by grant for PhD student, Dalal Al Ghanim, from King Saud University, Saudi Arabia.

\bibliographystyle{abbrvnat}
\bibliography{references}

\begin{thebibliography}{16}
\providecommand{\natexlab}[1]{#1}
\providecommand{\url}[1]{\texttt{#1}}
\expandafter\ifx\csname urlstyle\endcsname\relax
  \providecommand{\doi}[1]{doi: #1}\else
  \providecommand{\doi}{doi: \begingroup \urlstyle{rm}\Url}\fi

\bibitem[Albrecher and Hipp(2007)]{LundbergRiskProcessWithTax}
H.~Albrecher and C.~Hipp.
\newblock Lundberg's risk process with tax.
\newblock \emph{Bl. DGVFM}, 28\penalty0 (1):\penalty0 13--28, 2007.
\newblock ISSN 1864-0281.
\newblock \doi{10.1007/s11857-007-0004-4}.

\bibitem[Albrecher et~al.(2008)Albrecher, Renaud, and
  Zhou]{LevyInsuranceRiskProcessWithTax}
H.~Albrecher, J.-F. Renaud, and X.~Zhou.
\newblock A {L}\'evy insurance risk process with tax.
\newblock \emph{J. Appl. Probab.}, 45\penalty0 (2):\penalty0 363--375, 2008.
\newblock ISSN 0021-9002.
\newblock \doi{10.1239/jap/1214950353}.

\bibitem[Albrecher et~al.(2009)Albrecher, Borst, Boxma, and
  Resing]{TheTaxIdentity}
H.~Albrecher, S.~Borst, O.~Boxma, and J.~Resing.
\newblock The tax identity in risk theory---a simple proof and an extension.
\newblock \emph{Insurance Math. Econom.}, 44\penalty0 (2):\penalty0 304--306,
  2009.
\newblock ISSN 0167-6687.
\newblock \doi{10.1016/j.insmatheco.2008.05.001}.

\bibitem[Albrecher et~al.(2014)Albrecher, Avram, Constantinescu, and
  Ivanovs]{MAPtax2014}
H.~Albrecher, F.~Avram, C.~Constantinescu, and J.~Ivanovs.
\newblock The tax identity for {M}arkov additive risk processes.
\newblock \emph{Methodol. Comput. Appl. Probab.}, 16\penalty0 (1):\penalty0
  245--258, 2014.
\newblock ISSN 1387-5841.
\newblock \doi{10.1007/s11009-012-9310-y}.

\bibitem[Bogachev(2007)]{MeasureTheory}
V.~I. Bogachev.
\newblock \emph{Measure theory. {V}ol. {I}, {II}}.
\newblock Springer-Verlag, Berlin, 2007.
\newblock ISBN 978-3-540-34513-8; 3-540-34513-2.
\newblock \doi{10.1007/978-3-540-34514-5}.

\bibitem[Cheung and Landriault(2012)]{CheungCramer}
E.~C.~K. Cheung and D.~Landriault.
\newblock On a risk model with surplus-dependent premium and tax rates.
\newblock \emph{Methodol. Comput. Appl. Probab.}, 14\penalty0 (2):\penalty0
  233--251, 2012.
\newblock ISSN 1387-5841.
\newblock \doi{10.1007/s11009-010-9197-4}.

\bibitem[Jackiewicz(2009)]{GeneralLinearMethodsForODEs}
Z.~Jackiewicz.
\newblock \emph{General linear methods for ordinary differential equations}.
\newblock John Wiley \& Sons, Inc., Hoboken, NJ, 2009.
\newblock ISBN 978-0-470-40855-1.
\newblock \doi{10.1002/9780470522165}.

\bibitem[Kyprianou(2014)]{FluctuationBook}
A.~E. Kyprianou.
\newblock \emph{Fluctuations of {L}\'evy processes with applications}.
\newblock Universitext. Springer, Heidelberg, second edition, 2014.
\newblock ISBN 978-3-642-37631-3; 978-3-642-37632-0.
\newblock \doi{10.1007/978-3-642-37632-0}.
\newblock Introductory lectures.

\bibitem[Kyprianou and Ott(2012)]{KaprianouAndOtt}
A.~E. Kyprianou and C.~Ott.
\newblock Spectrally negative {L}\'evy processes perturbed by functionals of
  their running supremum.
\newblock \emph{J. Appl. Probab.}, 49\penalty0 (4):\penalty0 1005--1014, 2012.
\newblock ISSN 0021-9002.

\bibitem[Kyprianou and Zhou(2009)]{GeneralTaxStructure}
A.~E. Kyprianou and X.~Zhou.
\newblock General tax structures and the {L}\'evy insurance risk model.
\newblock \emph{J. Appl. Probab.}, 46\penalty0 (4):\penalty0 1146--1156, 2009.
\newblock ISSN 0021-9002.
\newblock \doi{10.1239/jap/1261670694}.

\bibitem[Li et~al.(2013)Li, Tang, and Zhou]{litangzhou}
B.~Li, Q.~Tang, and X.~Zhou.
\newblock A time-homogeneous diffusion model with tax.
\newblock \emph{J. Appl. Probab.}, 50\penalty0 (1):\penalty0 195--207, 2013.
\newblock ISSN 0021-9002.
\newblock \doi{10.1239/jap/1363784433}.

\bibitem[Price(1984)]{MultivariableAnalysis}
G.~B. Price.
\newblock \emph{Multivariable analysis}.
\newblock Springer-Verlag, New York, 1984.
\newblock ISBN 0-387-90934-6.
\newblock \doi{10.1007/978-1-4612-5228-3}.

\bibitem[Renaud(2009)]{TheDistributionOfTaxPayments}
J.-F. Renaud.
\newblock The distribution of tax payments in a {L}\'evy insurance risk model
  with a surplus-dependent taxation structure.
\newblock \emph{Insurance Math. Econom.}, 45\penalty0 (2):\penalty0 242--246,
  2009.
\newblock ISSN 0167-6687.
\newblock \doi{10.1016/j.insmatheco.2009.07.004}.

\bibitem[Revuz and Yor(1999)]{revuzyor}
D.~Revuz and M.~Yor.
\newblock \emph{Continuous martingales and {B}rownian motion}, volume 293 of
  \emph{Grundlehren der Mathematischen Wissenschaften [Fundamental Principles
  of Mathematical Sciences]}.
\newblock Springer-Verlag, Berlin, third edition, 1999.
\newblock ISBN 3-540-64325-7.
\newblock \doi{10.1007/978-3-662-06400-9}.

\bibitem[Wang and Hu(2012)]{OptimalLossCarryForwardWang}
W.~Wang and Y.~Hu.
\newblock Optimal loss-carry-forward taxation for the {L}\'evy risk model.
\newblock \emph{Insurance Math. Econom.}, 50\penalty0 (1):\penalty0 121--130,
  2012.
\newblock ISSN 0167-6687.
\newblock \doi{10.1016/j.insmatheco.2011.10.011}.

\bibitem[Wei(2009)]{wei}
L.~Wei.
\newblock Ruin probability in the presence of interest earnings and tax
  payments.
\newblock \emph{Insurance Math. Econom.}, 45\penalty0 (1):\penalty0 133--138,
  2009.
\newblock ISSN 0167-6687.
\newblock \doi{10.1016/j.insmatheco.2009.05.004}.

\end{thebibliography}

\end{document}